\documentclass[12pt]{amsart}
\makeatletter
\let\@@pmod\pmod
\DeclareRobustCommand{\pmod}{\@ifstar\@pmods\@@pmod}
\def\@pmods#1{\mkern4mu({\operator@font mod}\mkern 6mu#1)}
\makeatother
\usepackage{fullpage}
\usepackage{amssymb}
\usepackage{mathtools}
\usepackage{hyperref}
\usepackage{comment}
\newtheorem{theorem}{Theorem}
\newtheorem{lemma}[theorem]{Lemma}
\newtheorem{conjecture}[theorem]{Conjecture}
\newtheorem{corollary}[theorem]{Corollary}
\theoremstyle{remark}

\numberwithin{theorem}{section}
\numberwithin{equation}{section}
\newcommand{\C}{\mathbb{C}}

\newcommand{\Q}{\mathbb{Q}}
\newcommand{\R}{\mathbb{R}}
\newcommand{\Z}{\mathbb{Z}}
\newcommand{\N}{\mathbb{{N}}}
\newcommand{\HH}{\mathbb{{H}}}
\DeclareMathOperator{\GL}{GL}
\DeclareMathOperator{\SL}{SL}
\DeclareMathOperator{\tr}{tr}
\newcommand{\eul}{\varphi}
\newcommand{\eps}{\varepsilon}
\DeclareMathOperator{\lcm}{lcm}
\DeclareMathOperator{\ord}{ord}
\DeclareMathOperator{\height}{ht}
\def\sumprime{\operatornamewithlimits{\sum\nolimits^\prime}}

\newcommand{\sm}{\left(\begin{smallmatrix}}
\newcommand{\esm}{\end{smallmatrix}\right)}
\newcommand{\bpm}{\begin{pmatrix}}
\newcommand{\ebpm}{\end{pmatrix}}

\begin{document}
\title{A conjectural extension of Hecke's converse theorem}
\author{Sandro Bettin}
\author{Jonathan W. Bober}
\author{Andrew R. Booker}
\author{Brian Conrey}
\author{Min Lee}
\author{Giuseppe Molteni}
\author{Thomas Oliver}
\author{David J. Platt}
\author{Raphael S. Steiner}
\address{(S.~B.) Dipartimento di Matematica\\Universit\`a di Genova\\via Dodecaneso 35\\16146 Genova\\Italy}
\address{(J.~W.~B.\ and T.~O.) School of Mathematics, University of Bristol, Bristol, BS8 1TW, UK, and the Heilbronn Institute for Mathematical Research, Bristol, UK}
\address{(A.~R.~B., M.~L., D.~J.~P.\ and R.~S.~S.) School of Mathematics, University of Bristol, Bristol, BS8 1TW, UK}
\address{(B.~C.) School of Mathematics, University of Bristol, Bristol, BS8 1TW, UK, and the American Institute of Mathematics, San Jose, CA}
\address{(G.~M.) Dipartimento di Matematica\\Universit\`a di Milano\\via Saldini 50\\20133 Milano\\Italy}
\thanks{S.~B.\ and G.~M.\ were partially supported by PRIN ``Number Theory and Arithmetic Geometry". A.~R.~B., M.~L.\ and D.~J.~P.\ were partially supported by EPSRC Grant \texttt{EP/K034383/1}.}
\email{\texttt{andrew.booker@bristol.ac.uk} (corresponding author)}
\begin{abstract}
We formulate a precise conjecture that, if true, extends the converse
theorem of Hecke without requiring hypotheses on twists by Dirichlet
characters or an Euler product. The main idea is to linearize the
Euler product, replacing it by twists by Ramanujan sums.
We provide evidence for the conjecture, including proofs of some special
cases and under various additional hypotheses.
\end{abstract}
\maketitle

\section{Introduction}
Let $f\in M_k(\Gamma_0(N),\xi)$ be a classical holomorphic modular form
of weight $k$, level $N$ and nebentypus character $\xi$, and define
\begin{equation}\label{eq:gdef}
g(z)=(\sqrt{N}z)^{-k}f\!\left(-\frac1{Nz}\right).
\end{equation}
Let $f_n$ and $g_n$ denote
the Fourier coefficients of $f$ and $g$, respectively, and define
\begin{equation}\label{eq:Lambdaf}
\Lambda_f(s)=\Gamma_\C(s+\tfrac{k-1}2)\sum_{n=1}^\infty f_nn^{-s-\frac{k-1}2}
\quad\text{and}\quad
\Lambda_g(s)=\Gamma_\C(s+\tfrac{k-1}2)\sum_{n=1}^\infty g_nn^{-s-\frac{k-1}2}
\end{equation}
for $\Re(s)>\frac{k+1}2$, where $\Gamma_\C(s):=2(2\pi)^{-s}\Gamma(s)$.
Then $\Lambda_f(s)$ and $\Lambda_g(s)$
continue to entire functions of finite order, apart from at most simple
poles at $s=\frac{1\pm k}2$, and satisfy the functional equation
\begin{equation}\label{eq:fe}
\Lambda_f(s)=i^kN^{\frac12-s}\Lambda_g(1-s).
\end{equation}

Conversely, when $N\le 4$, Hecke \cite{Hecke,Hecke2} (see
also~\cite{BerndtKnopp}) showed that the modular
forms of level $N$ are characterized by these properties. Precisely,
given sequences
$\{f_n\}_{n=1}^\infty$, $\{g_n\}_{n=1}^\infty$ of at most polynomial growth,
if the functions $\Lambda_f(s)$ and $\Lambda_g(s)$ defined by
\eqref{eq:Lambdaf} continue to entire functions
of finite order and satisfy \eqref{eq:fe} then $f_n$ and $g_n$ are the
Fourier coefficients of modular forms of level $N$ and weight $k$,
related by \eqref{eq:gdef}.

When $N\ge5$, Hecke's proof no longer goes through, and in fact
the vector space of sequences $\{f_n\}_{n=1}^\infty$,
$\{g_n\}_{n=1}^\infty$ satisfying the above conditions is infinite
dimensional.
Weil~\cite{weil} showed that one can recover the converse statement by
assuming additional functional equations for twisted $L$-functions
\begin{equation}\label{eq:Lfchidef}
\Lambda_f(s,\chi)=\Gamma_\C(s+\tfrac{k-1}2)\sum_{n=1}^\infty f_n\chi(n)n^{-s-\frac{k-1}2}
\end{equation}
for primitive characters $\chi$ of conductor coprime to $N$.
On the other hand, it has been conjectured (see
\cite[Conjecture~1.2]{fkl}) that
if $\Lambda_f(s)$ and $\Lambda_g(s)$ have Euler product
expansions\footnote{We regard the factors of $\Gamma_\C(s+\frac{k-1}2)$
in \eqref{eq:Lambdaf} as Euler factors for the archimedean place.}
of the shape satisfied by primitive Hecke eigenforms then the single
functional equation \eqref{eq:fe} should suffice to imply modularity,
without the need for character twists. Some partial progress on
this problem was made by Conrey and Farmer \cite{conrey-farmer} (see
also~\cite{ConreyFarmerOdgersSnaith}), who proved the conjecture for
some values of $N$ exceeding $4$.

One drawback of assuming an Euler product is that it imposes a nonlinear
constraint on the Fourier coefficients $f_n,g_n$,
so the solutions to \eqref{eq:fe} no longer
form a vector space. In turn, it is unclear how to make use of
this constraint to extend Hecke's proof to higher level.
In this paper we propose a replacement for the Euler product
that, we conjecture, characterizes the modular forms
of any level $N$, yet retains the linearity of \eqref{eq:fe}:
\begin{conjecture}\label{main}
Let $\xi$ be a Dirichlet character modulo $N$, $k$ a positive integer
satisfying $\xi(-1)=(-1)^k$, and
$\{f_n\}_{n=1}^\infty,\{g_n\}_{n=1}^\infty$ sequences of complex numbers
satisfying $f_n,g_n=O(n^\sigma)$ for some $\sigma>0$.
For $q\in\N$, let
$$
c_q(n)=\sum_{\substack{a\pmod*{q}\\(a,q)=1}}e\!\left(\frac{an}{q}\right)
$$
be the associated Ramanujan sum, where $e(x):=e^{2\pi ix}$, and define
$$
\Lambda_f(s,c_q)=\Gamma_\C\bigl(s+\tfrac{k-1}2\bigr)\sum_{n=1}^\infty
\frac{f_nc_q(n)}{n^{s+\frac{k-1}2}}
\quad\text{and}\quad
\Lambda_g(s,c_q)=\Gamma_\C\bigl(s+\tfrac{k-1}2\bigr)\sum_{n=1}^\infty
\frac{g_nc_q(n)}{n^{s+\frac{k-1}2}}
$$
for $\Re(s)>\sigma+1-\frac{k-1}2$.
For every $q$ coprime to $N$, suppose that $\Lambda_f(s,c_q)$ and
$\Lambda_g(s,c_q)$ continue to entire functions of
finite order and satisfy the functional equation
\begin{equation}\label{e:fe_cq}
	\Lambda_f(s,c_q)=i^k \xi(q)(Nq^2)^{\frac12-s}\Lambda_g(1-s,c_q).
\end{equation}
Then $f(z):=\sum_{n=1}^\infty f_ne(nz)$ is an element of
$M_k(\Gamma_0(N),\xi)$.
\end{conjecture}

To understand the motivation behind this conjecture, we first consider
a more general family of twists.  Let $\chi\pmod*{q}$ be a Dirichlet
character, not necessarily primitive, and define
\begin{equation}\label{e:cchi}
	c_\chi(n)
	=
	\sum_{\substack{a\pmod*{q}\\(a,q)=1}}
	\chi(a)e\!\left(\frac{an}q\right),
\end{equation}
\begin{equation}\label{e:Lambda}
	\Lambda_f(s,c_\chi)=\Gamma_\C(s+\tfrac{k-1}2)
	\sum_{n=1}^\infty\frac{f_nc_\chi(n)}{n^{s+\frac{k-1}2}}
	\quad\text{and}\quad
	\Lambda_g(s,c_{\overline{\chi}})=\Gamma_\C(s+\tfrac{k-1}2)
	\sum_{n=1}^\infty\frac{g_nc_{\overline{\chi}}(n)}{n^{s+\frac{k-1}2}}.
\end{equation}
Note that when $\chi$ is the trivial character mod $q$, $c_\chi$ reduces
to the Ramanujan sum, $c_q$.
In Lemma~\ref{lem:Lambda_chi}, we show that if we start from a pair of modular forms
$f,g$ satisfying \eqref{eq:gdef}, then $\Lambda_f(s,c_\chi)$ and
$\Lambda_g(s,c_{\overline{\chi}})$
satisfy the functional equation
\begin{equation}\label{eq:fe_cchi}
	\Lambda_f(s,c_\chi)
	=
	i^k\xi(q)\overline{\chi(-N)}(Nq^2)^{\frac12-s}
	\Lambda_g(1-s,c_{\overline{\chi}}).
\end{equation}
When $\chi$ is primitive, we have
$c_\chi(n)=\tau(\chi)\overline{\chi(n)}$, where
$\tau(\chi)=\sum_{a=1}^q\chi(a)e(a/q)$ denotes the Gauss sum,
and \eqref{eq:fe_cchi} reduces to the familiar functional equation
for the multiplicative twist $\Lambda_f(s,\overline{\chi})$.
More generally,
when $\Lambda_f(s)$ possesses an Euler product, we show in
Lemma~\ref{lem:L_twist_Rsum} that \eqref{eq:fe_cchi} is implied by the functional
equation for $\Lambda_f(s,\overline{\chi}_*)$, where $\chi_*$ is the
primitive character inducing $\chi$.
In particular, in the presence of an Euler product, \eqref{eq:fe} implies
\eqref{e:fe_cq}.

Given any $Q\in\N$ and $q\mid Q$, we can view $c_\chi$ for $\chi\pmod*{q}$
as a function on $\Z/Q\Z$. One can show that as $\chi$ ranges over
all characters of modulus dividing $Q$, the functions $c_\chi$ form
an orthogonal basis for the space of functions on $\Z/Q\Z$. Thus, any
twist of $f$ with periodic coefficients and period coprime to $N$ is
a linear combination of the twists by $c_\chi$. In this sense,
\eqref{eq:fe_cchi}
is the most general functional equation (from twists with period coprime
to the level) that one can expect.

Conjecture~\ref{main} arises from the speculation that any constraints
on the solutions to \eqref{eq:fe} imposed by the assumption of an
Euler product are already implied by the extra functional equations
\eqref{eq:fe_cchi} that one obtains from taking $\chi$ equal to the
trivial character mod $q$. In Section~\ref{sec:results}, we prove five
theorems that lend some support to the conjecture:
\begin{enumerate}
\item Theorem~\ref{thm:smallN} establishes Conjecture~\ref{main} for
some values of $N$ exceeding $4$, following the methods of Conrey and
Farmer \cite{conrey-farmer}.
\item Theorem~\ref{thm:finiteindex} proves Conjecture~\ref{main} under
the additional assumption that $f$ is modular for some subgroup of
finite index in $\SL_2(\Z)$ (not necessarily a congruence subgroup).
\item Theorem~\ref{thm:absolutef} proves Conjecture~\ref{main} under the
additional assumption that $|f|$ is modular for some congruence
subgroup.
\item Theorem~\ref{thm:commutator} proves Conjecture~\ref{main} under
the additional assumptions that $N$ is prime and $f$ is modular for the
commutator subgroup of $\Gamma_0(N)$. This establishes a version of
Theorem~\ref{thm:finiteindex} for some cases of infinite index.
\item Theorem~\ref{thm:singleq} shows that for almost all primes $q$,
the hypotheses of Conjecture~\ref{main}, together with the expected
analytic properties and functional equations of the multiplicative
character twists \eqref{eq:Lfchidef} for the primitive characters
$\chi\pmod*{q}$, suffice to imply modularity. Particular examples of
suitable $q$ are given for some levels outside the scope of
Theorem~\ref{thm:smallN}.
\end{enumerate}
To set these results in context, we note that one reason why Hecke's
argument fails for $N\ge5$ is that there are counterexamples arising from
more general kinds of modular forms. If one believes that a twistless
converse theorem is possible assuming an Euler product, then it is
reasonable to ask how these counterexamples are eliminated by the Euler
product.  Points (2) and (3) above address two such generalizations of
modular forms, namely forms for noncongruence groups and forms for more
general weight-$k$ multiplier systems (not necessarily of finite order).

Concerning point (5), Diaconu, Perelli and Zaharescu \cite{DPZ} showed
that if $\Lambda_f(s)$ is given by an Euler product, then there exists
a prime $q$ (depending on $N$) such that the analytic properties
and functional equations of the character twists \eqref{eq:Lfchidef}
for all primitive $\chi$ of conductor dividing $q$ suffice to imply
modularity. On the other hand, again under the assumption of an Euler
product, it follows from a theorem of Piatetski-Shapiro \cite{PS} that
it suffices to assume the expected properties of \eqref{eq:Lfchidef} for
all primitive $\chi\pmod*{p^j}$ for any fixed prime $p$ and all $j\ge0$.
Point (5) can be seen as a complement to both of these results. We
conjecture that the proof of Theorem~\ref{thm:singleq}
can be extended to all sufficiently large primes $q$, and we study this
problem in detail in Section~\ref{sec:generating}.

\subsection*{Acknowledgements}
This paper grew out of a focused research workshop on the \emph{Sarnak
rigidity conjecture} at the Heilbronn Institute for Mathematical
Research. We thank the Institute for their support, which made this
work possible.

\section{Main results}\label{sec:results}
Let $\HH=\{z\in\C:\Im(z)>0\}$ denote the upper half-plane.
For any function $h:\HH\to\C$ and any matrix
$\gamma=\sm a&b\\c&d\esm\in\GL_2^+(\R)=\{M\in\GL_2(\R):\det{M}>0\}$,
define
$$
h|\gamma=(\det\gamma)^{k/2}(cz+d)^{-k}h\!\left(\frac{az+b}{cz+d}\right),
$$
where $k\in\N$ is the integer appearing in Conjecture~\ref{main}.
(We assume that $k$ is fixed from now on and suppress it from the
notation.)
Note that this defines a right action, i.e.\ 
$h|(\gamma_1\gamma_2)=(h|\gamma_1)|\gamma_2$ for any
$\gamma_1,\gamma_2\in\GL_2^+(\R)$.
We extend the action linearly to the group algebra $\C[\GL_2^+(\R)]$,
i.e.\ for $\gamma=\sum_i c_i\gamma_i\in\C[\GL_2^+(\R)]$ we define
$h|\gamma=\sum_i c_ih|\gamma_i$.

Let $f$ be as in Conjecture~\ref{main}, and define
$g(z)=\sum_{n=1}^\infty g_ne(nz)$.
Then, by Hecke's argument \cite[Theorem~4.3.5]{miyake},
the fact that $\Lambda_f(s,c_1)$
and $\Lambda_g(s,c_1)$ continue to entire functions of finite order and
satisfy \eqref{e:fe_cq} for $q=1$ is equivalent to the identity
$f|\sm&-1\\N&\esm=g$.
Writing $T=\sm1&1\\&1\esm$ and
$W=\sm&-1\\N&\esm T^{-1}\sm&-1\\N&\esm^{-1}=\sm1&\\N&1\esm$,
since $f$ and $g$ are given by Fourier series,
we have $f|T=f|W=f$.

Given a matrix $\gamma=\sm a&b\\c&d\esm\in\Gamma_0(N)$, we define
$\xi(\gamma)=\xi(d)$. Since $\xi(-1)=(-1)^k$, we have
$f|(-I)=\xi(-I)f$, and thus $f|\gamma=\xi(\gamma)f$ for every
$\gamma\in\langle-I,T,W\rangle$. To prove that $f\in
M_k(\Gamma_0(N),\xi)$, it suffices to verify this equality for every
$\gamma\in\Gamma_0(N)$, since the holomorphy of $f$ at cusps follows
from modularity and the growth estimate $f_n=O(n^\sigma)$.

Note that if $\gamma,\gamma'\in\Gamma_0(N)$ have the same top row
then $\gamma'\gamma^{-1}$ is a power of $W$, so that
$f|\gamma'=f|\gamma$. Thus, $f|\gamma$ depends only on the top row of
$\gamma$. With this in mind, we will write $\gamma_{q,a}$ to denote any
element of $\Gamma_0(N)$ with top row $\sm q&-a\esm$.

\begin{theorem}\label{thm:smallN}
Conjecture~\ref{main} is true for $N\le9$ and $N\in\{11,15,17,23\}$.
\end{theorem}
\begin{proof}
The following table shows, for each $N$ in the statement of the theorem,
minimal generating sets for $\Gamma_0(N)$, verified with Sage \cite{sage}:
\begin{center}
\begin{tabular}{rl|rl}
$N$ & generators & $N$ & generators \\ \hline
$1$ & $\{T,W\}$ &
$8$ & $\{-I,T,W,\gamma_{3,1}\}$ \\
$2$ & $\{T,W\}$ &
$9$ & $\{-I,T,W,\gamma_{2,1}\}$ \\
$3$ & $\{T,-W\}$ &
$11$ & $\{-I,W,\gamma_{2,1},\gamma_{3,1}\}$\\
$4$ & $\{-I,T,W\}$ &
$15$ & $\{-I,T,W,\gamma_{2,1},\gamma_{4,1},\gamma_{11,4}\}$\\
$5$ & $\{T,W,\gamma_{2,1}\}$ &
$17$ & $\{T,W,\gamma_{2,1},\gamma_{3,1},\gamma_{6,1}\}$\\
$6$ & $\{-I,T,W,\gamma_{5,2}\}$ &
$23$ & $\{-I,T,W,\gamma_{2,1},\gamma_{4,1},\gamma_{6,1},\gamma_{10,-3}\}$\\
$7$ & $\{T,W,-\gamma_{2,1}\}$
\end{tabular}
\end{center}
In particular, for $N\le 4$, $\Gamma_0(N)$ is generated by
$-I$, $T$ and $W$, so there is nothing to prove. For all other levels
we apply the methods of Conrey and Farmer \cite{conrey-farmer}, in the
form of Lemmas~\ref{lem:cq}, \ref{lem:conrey-farmer} and \ref{lem:1qinv}.

For odd values of $N$, Lemma~\ref{lem:cq} with $q=2$ implies that
$f|\gamma_{2,1}=\overline{\xi(2)}f$. In view of
the table, this establishes the claim for
$N\in\{5,7,9\}$.

For $N\in\{8,11,15,17,23\}$ we obtain values of $q\in\{3,4,6\}$
for which $f|\gamma_{q,1}=\overline{\xi(q)}f$ from
Lemma~\ref{lem:conrey-farmer}. For $N\in\{8,11,17\}$ these are
sufficient to establish the claim.

It remains only to prove the claim for $N=6,15,23$, for which we need to
show modularity with respect to the generators
$\gamma_{5,2}$, $\gamma_{11,4}$, $\gamma_{10,-3}$, respectively.
For $N=6$ we have the equalities
$$
\bpm 5&-1\\6&-1\ebpm=-TW^{-1}
\quad\text{and}\quad
\bpm 5&1\\-6&-1\ebpm=-T^{-1}W,
$$
so Lemma~\ref{lem:cq} with $q=5$ takes the form
$$
f\biggl|\bigl[\gamma-\xi(-1)\bigr]\bpm 1&2/5\\&1\ebpm
+f\biggl|\bigl[\gamma^{-1}-\xi(-1)\bigr]\bpm 1&-2/5\\&1\ebpm
=0,
$$
where $\gamma=\sm 5&-2\\-12&5\esm$.
Applying Lemma~\ref{lem:1qinv} with $\alpha=4/5$ and $\zeta=-1$, we obtain
$f|\gamma=\xi(-1)f$.

For $N=15$ we have the equalities
$$
\bpm 8&-1\\-15&2\ebpm
=T^{-1}\bpm 2&-1\\15&-7\ebpm^{-1},
\quad
\bpm 8&1\\15&2\ebpm
=\bpm 2&-1\\-15&8\ebpm^{-1},
$$
$$
\bpm 8&-3\\75&-28\ebpm
=-\bpm 2&-1\\15&-7\ebpm T\bpm 11&-4\\-30&11\ebpm,
\quad
\bpm 8&3\\45&17\ebpm
=-\bpm 2&-1\\15&-7\ebpm T\bpm 11&-4\\-30&11\ebpm^{-1},
$$
so Lemma~\ref{lem:cq} with $q=8$ takes the form
$$
\xi(7)f\biggl|\bigl[\gamma-\xi(11)\bigr]\bpm 1&3/8\\&1\ebpm
+\xi(7)f\biggl|\bigl[\gamma^{-1}-\xi(11)\bigr]\bpm 1&-3/8\\&1\ebpm
=0,
$$
where $\gamma=\sm 11&-4\\-30&11\esm$.
Applying Lemma~\ref{lem:1qinv} with $\alpha=3/4$ and $\zeta=-1$, we obtain
$f|\gamma=\xi(11)f$.

For $N=23$ we have the equalities
$$
\bpm 3&-1\\-23&8\ebpm
=-\bpm 4&-1\\-23&6\ebpm
\bpm 6&-1\\-23&4\ebpm^{-1}
\bpm 10&3\\23&7\ebpm^{-1}
$$
and
$$
\bpm 3&1\\23&8\ebpm
=-\bpm 2&-1\\23&-11\ebpm
\bpm 10&3\\23&7\ebpm,
$$
so Lemma~\ref{lem:cq} with $q=3$ takes the form
$$
\xi(11)f\biggl|\bigl[\gamma-\xi(7)\bigr]\bpm 1&-1/3\\&1\ebpm
+\xi(10)f\biggl|\bigl[\gamma^{-1}-\xi(10)\bigr]\bpm 1&1/3\\&1\ebpm
=0,
$$
where $\gamma=\sm 10&3\\23&7\esm$.
Applying Lemma~\ref{lem:1qinv} with $\alpha=-2/3$ and $\zeta=-\xi(8)$,
we obtain $f|\gamma=\xi(7)f$.
\end{proof}

\begin{theorem}\label{thm:finiteindex}
Assume the hypotheses of Conjecture~\ref{main}.
Suppose that there is a subgroup $H<\Gamma_1(N)$ of finite index such
that $f|\gamma=f$ for all $\gamma\in H$. Then
$f\in M_k(\Gamma_0(N),\xi)$.
\end{theorem}
\begin{proof}
We may assume without loss of generality that $H$ contains $T$ and $W$.
By Lemma~\ref{lem:cq}, for any prime $q\nmid N$, 
\begin{equation}\label{e:gz_cq}
	\sum_{a=1}^{q-1} f\bigg|
	\Bigl[\gamma_{q,a}-\overline{\xi(q)}\Bigr] \bpm 1 & \frac{a}{q} \\ 0 & 1\ebpm
	=0.
\end{equation}

Put $h=[\Gamma_0(N):H]$, and let 
$g_1,\ldots,g_h\in\Gamma_0(N)$ be coset representatives
for $H\backslash\Gamma_0(N)$. Replacing $g_i$ by $Wg_i$ if
necessary, we may assume without loss of generality that $g_i$ is
not upper triangular.
For each $\gamma_{q,a}\in \Gamma_0(N)$, 
there exists $i\in\{1,\ldots,h\}$ such that 
$\gamma_{q,a}\in H g_i$, so that $f|\gamma_{q,a}=f|g_i$.
Rearranging \eqref{e:gz_cq}, we get
$$
	\sum_{i=1}^h f\bigg| 
	\bigl[g_i - \xi(g_i)\bigr]
	\sum_{\ell=1}^{\kappa_i} \bpm 1 & \frac{a_{i\ell}}{q}\\ 0 & 1\ebpm 
	=0,
$$
where $\bigcup_{i=1}^h\{a_{i\ell}:\ell=1,\ldots,\kappa_i\}$ is a
disjoint partition of $\{1,\ldots,q-1\}$.

For each $i\in\{1,\ldots,h\}$, since
$[\Gamma_0(N): g_i^{-1} H g_i]=[\Gamma_0(N):H]<\infty$,
there exists $m_i\in\N$ such that 
$$
	g_i^{-1} H g_iT^{m_i} = g_i^{-1} H g_i. 
$$
Setting $m=\lcm(m_1,\ldots,m_h)$, we have
$g_i T^m \in H g_i$ for all $i$. 
Then $f|g_i T^{m} = f|g_i$, and thus
$f|[g_i-\xi(g_i)]$ has a Fourier expansion: 
\begin{equation}\label{e:gammai_Fourier}
	f\bigl|\bigl[g_i -\xi(g_i)\bigr]
	= \sum_{n\in\Z}\lambda_i(n)e\!\left(n\tfrac{z}{m}\right). 
\end{equation}
Therefore,
$$
	\sum_{i=1}^h f\bigg| \bigl[g_i-\xi(g_i)\bigr] 
	\sum_{\ell=1}^{\kappa_i} \bpm 1 & \tfrac{a_{i\ell}}{q}\\ 0 & 1\ebpm 
	=
	\sum_{n\in\Z}
	\sum_{i=1}^h \lambda_i(n) 
	\left(\sum_{\ell=1}^{\kappa_i} e\!\left(n \tfrac{a_{i\ell}}{qm}\right)\right)
	e\!\left(n\tfrac{z}{m}\right)
	=0,  
$$
i.e., for $n\in\Z$, 
\begin{equation}\label{e:linear0}
	\sum_{i=1}^h \lambda_i(n) \left(\sum_{\ell=1}^{\kappa_i} e\!\left(n \tfrac{a_{i\ell}}{qm}\right)\right)
	=
	0.
\end{equation}

Fix $n\in\Z\setminus\{0\}$.
By Dirichlet's theorem, we can choose distinct primes
$q_1,\ldots,q_h\nmid mnN$ and integers $a_1,\ldots,a_h$ such that
$\gamma_{q_i,a_i}\in\langle T\rangle g_i\subseteq Hg_i$ for each $i$.
Thus, from \eqref{e:linear0} for $q\in\{q_1,\ldots,q_h\}$, we
obtain a system of linear equations of the shape
\begin{equation}\label{e:linear}
	\sum_{i=1}^h \left(\sum_{\ell=1}^{\kappa_{i,j}}
	e\!\left(n\tfrac{a^{(j)}_{i\ell}}{q_j m}\right)\right)\lambda_i(n)=0
\quad\text{for }j\in\{1,\ldots,h\},
\end{equation}
with $\kappa_{i,i}>0$ for every $i\in\{1,\ldots,h\}$.
By Lemma~\ref{lem:key}, 
$$
	\det\left(\left[\sum_{\ell=1}^{\kappa_{i, j}} 
	e\!\left(n\tfrac{a^{(j)}_{i\ell}}{q_j m}\right)\right]_{1\leq i, j\leq h}\right)
	\neq 0, 
$$
so \eqref{e:linear} has only the trivial solution
$\lambda_1(n)=\ldots=\lambda_h(n)= 0$. 

Since $n\in\Z\setminus\{0\}$ was arbitrary,
it follows from \eqref{e:gammai_Fourier} that
$f|[g_i-\xi(g_i)]$ is a constant, say $C_i$.
Since $g_i^{-1}Hg_i\cap H$ has finite index in $\SL_2(\Z)$,
there exists
$\gamma=\sm a&b\\c&d\esm\in g_i^{-1}Hg_i\cap H$ with $c\ne0$.
Then $C_i=C_i|\gamma=(cz+d)^{-k}C_i$. Since $k>0$, we must have $C_i=0$,
i.e.\ $f|g_i=\xi(g_i)f$. This concludes the proof.
\end{proof}

\begin{theorem}\label{thm:absolutef}
Assume the hypotheses of Conjecture~\ref{main}, and suppose that
there is a congruence subgroup $H<\Gamma_0(N)$ such that
$\bigl|(f|\gamma)(z)\bigr|=|f(z)|$ for all $\gamma\in H$.
Then $f\in M_k(\Gamma_0(N),\xi)$.
\end{theorem}
\begin{proof}
If $f=0$ then the conclusion is trivially true, so from now on assume $f\neq 0$.
Let $M$ denote the level of $H$, so that $H\supseteq\Gamma(M)$.
Since $f|T=f|W=f$ and $\Gamma_1(N)$ is generated by
$\{T,W\}\cup\Gamma(M)$,
we may assume without loss of generality that
$H\supseteq\Gamma_1(N)$.
By Theorem~\ref{thm:almostall}, there exists a prime
$q\equiv1\pmod*{N}$ such that $\Gamma_1(N)$ is generated by
$\{T,W,\gamma_{q,a}:1\le a<q\}$.
By Lemma~\ref{lem:nonvanishing}, there exists $m\in\N$ such that $q\mid m$ and
$\{f_m,g_m\}\ne\{0\}$. Since $\sm &-1\\N&\esm$ normalizes $\Gamma_1(N)$,
we may swap the roles of $f$ and $g$ if necessary, so as to assume that
$f_m\ne0$.

For any $\gamma\in\Gamma_1(N)$, the function
$(f|\gamma)(z)/f(z)$ is meromorphic on $\HH$
and has modulus $1$; by the maximum modulus principle, it must be
a constant, say $\epsilon(\gamma)$.
By Lemma~\ref{lem:cq}, we have
$$
0=
\sum_{\substack{a\pmod*{q}\\(a,q)=1}}
f\left|\bigl[\gamma_{q,a}-1\bigr]\bpm 1&a/q\\&1\ebpm\right.
=\sum_{\substack{a\pmod*{q}\\(a,q)=1}}
\bigl[\epsilon(\gamma_{q,a})-1\bigr]f\left|\bpm 1&a/q\\&1\ebpm\right..
$$
Considering the Fourier expansion, this implies that
$$
\sum_{\substack{a\pmod*{q}\\(a,q)=1}}
\bigl[\epsilon(\gamma_{q,a})-1\bigr]
f_ne\!\left(\frac{an}q\right)=0
\quad\text{for all }n.
$$
In particular, taking $n=m$, we have
$$
\sum_{\substack{a\pmod*{q}\\(a,q)=1}}
\bigl[\epsilon(\gamma_{q,a})-1\bigr]=0,
$$
and since $|\epsilon(\gamma_{q,a})|=1$ for every $a$, it follows that
$\epsilon(\gamma_{q,a})=1$.
Therefore, $f|\gamma=f$ for all $\gamma\in\Gamma_1(N)$. Applying
Theorem~\ref{thm:finiteindex} with $H=\Gamma_1(N)$,
we conclude that $f\in M_k(\Gamma_0(N),\xi)$.
\end{proof}

\begin{theorem}\label{thm:commutator}
Assume the hypotheses of Conjecture~\ref{main}.
Suppose that $N$ is prime and that
$f|\gamma_1\gamma_2=f|\gamma_2\gamma_1$ for every pair
$\gamma_1,\gamma_2\in\Gamma_0(N)$. Then $f\in M_k(\Gamma_0(N),\xi)$.
\end{theorem}
\begin{proof}
Let $H$ be the smallest subgroup of $\Gamma_0(N)$ containing
$T$, $W$ and all commutators $\gamma_1\gamma_2\gamma_1^{-1}\gamma_2^{-1}$
for $\gamma_1,\gamma_2\in\Gamma_0(N)$. Then $H$ is a normal subgroup
with abelian quotient $H\backslash\Gamma_0(N)$,
and $f|\gamma=f$ for all $\gamma\in H$.
If $N\in\{2,3\}$ then $\langle H,-I\rangle=\Gamma_0(N)$ and there
is nothing to prove, so we assume henceforth that $N\ge5$.

Let $R=\{r\in\Z:2\le|r|<\frac12N\}$, and for each $r\in R$,
fix a matrix $\gamma_{r,1}$ with top row $\sm r&-1\esm$.
Then, by Lemma~\ref{lem:log-gen}, for any prime $q\nmid N$ and $a$
coprime to $q$, we have
$$
\gamma_{q,a}=\pm\prod_{i=1}^l\tau_i,
$$
where each $\tau_i$ is an element of
$\{T,T^{-1},W,W^{-1},\gamma_{r,1}^{-1}:r\in R\}$.
Since $H\backslash\Gamma_0(N)$ is abelian, we are free to permute the
$\tau_i$ without changing the coset $H\prod\tau_i$.
Hence, since $H$ contains $\langle T,W\rangle$, we may write
$$
H\gamma_{q,a}=H(-I)^\epsilon\prod_{r\in R}\gamma_{r,1}^{-e_r},
$$
for some $\epsilon\in\{0,1\}$ and non-negative integers $e_r$ (depending
on $q$ and $a$), satisfying $\sum_{r\in R}e_r\le\log_2q$.

Now, fix $s\in R$, $n\in\Z\setminus\{0\}$ and $X\in\N$,
and let $Q=Q(s,n,X)$ denote the set of primes $q$
satisfying $qs\equiv 1\pmod*{N}$, $q\nmid n$ and $q\le X$. As in the
proof of Theorem~\ref{thm:finiteindex}, we consider \eqref{e:gz_cq}
for all primes $q\in Q$. Let $g_1,\ldots,g_h$ be a minimal set
of representatives for the cosets $H\gamma_{q,a}$ of all matrices
occurring there.  By the above, we may take each $g_i$ of the form
$(-I)^\epsilon\prod_{r\in R}\gamma_{r,1}^{-e_r}$ with
$\epsilon\in\{0,1\}$, $e_r\ge0$ and
$\sum_{r\in R}e_r\le\log_2X$.
In particular, $H\gamma_{s,1}^{-1}=H\gamma_{q,-1}$ for every $q\in Q$, so we
may take $g_1=\gamma_{s,1}^{-1}$.  By Dirichlet's theorem, we have
$\#Q\gg X/\log{X}$, and thus $h\le2(1+\log_2X)^{N-3}\le\#Q$ for all
sufficiently large $X$.

For each $i\in\{1,\ldots,h\}$, we have $f|g_iT=f|Tg_i=f|g_i$, so
$f|[g_i-\xi(g_i)]$ has a Fourier expansion as in
\eqref{e:gammai_Fourier}, with $m=1$. In turn, this leads to the system
of linear equations \eqref{e:linear}, where we take
$\{q_j\}$ to be any subset of $Q$ of cardinality $h$.
Applying Lemma~\ref{lem:diagonal}, by appropriate permutation of the rows and
columns we can select a square subsystem for which the diagonal
entries are non-zero. Since the coset $Hg_1$ occurs in
every row, the column $i=1$ is necessarily one of the variables in
the subsystem.

Hence, by Lemma~\ref{lem:key}, we have $\lambda_1(n)=0$.
Since $n\in\Z\setminus\{0\}$ was arbitrary, we thus have that
$f|[\gamma_{s,1}^{-1}-\xi(s)]$ is a constant, say $C$.
Clearly $C|\gamma=C$ for all
$\gamma\in\gamma_{s,1}H\gamma_{s,1}^{-1}\cap H=H$.
Taking $\gamma=W$, it follows that
$C=0$, whence $f|\gamma_{s,1}^{-1}=\xi(s)f$.
Finally, Lemma~\ref{lem:log-gen} implies that $\Gamma_0(N)$
is generated by $-I$, $T$, $W$ and $\gamma_{s,1}$ for $s\in R$, so
$f|\gamma=\xi(\gamma)f$ for all $\gamma\in\Gamma_0(N)$.
\end{proof}

\begin{theorem}\label{thm:singleq}
Assume the hypotheses of Conjecture~\ref{main}. There is a set $Q$
of prime numbers such that
\begin{itemize}
\item[(i)]$Q$ has density $1$ in the set of all primes, and
\item[(ii)]if there exists $q\in Q$ such that
the multiplicative twists $\Lambda_f(s,\chi)$
and $\Lambda_g(s,\overline{\chi})$, for all primitive characters
$\chi\pmod*{q}$, continue to entire functions of finite order and satisfy
the functional equation
\begin{equation}\label{eq:femult}
\Lambda_f(s,\chi)=i^k\xi(q)\chi(N)q^{-1}\tau(\chi)^2(Nq^2)^{\frac12-s}
\Lambda_g(1-s,\overline{\chi}),
\end{equation}
then $f\in M_k(\Gamma_0(N),\xi)$.
\end{itemize}
In particular, for each $N$ in the following table, the set $Q$
contains every prime $q\nmid N$ in the indicated interval.
\begin{center}
\begin{tabular}{rr|rr}
$N$ & $q$ &
$N$ & $q$\\ \hline
$10$ & $(11,10^9)$ &
$18$ & $(53,10^9)$\\
$12$ & $(35,10^9)$ &
$19$ & $(37,10^9)$\\
$13$ & $(5,10^9)$ &
$20$ & $(79,10^9)$\\
$14$ & $(43,10^9)$ &
$21$ & $(83,10^9)$\\
$16$ & $(47,10^9)$ &
$22$ & $(43,10^9)$\\
\end{tabular}
\end{center}
\end{theorem}
\begin{proof}
Let $Q$ be the set of primes $q\nmid N$ such that
$H_q\supseteq\Gamma_1(N)$, in the notation of
Section~\ref{sec:generating}. By Theorem~\ref{thm:almostall},
$Q$ has density $1$ in the set of all primes, so (i) holds, and the fact
that $Q$ contains the numbers indicated in the table is the content of
Theorem~\ref{thm:explicitq}.

Let $q\in Q$.
Then by \cite[Lemmas~4.3.9 and 4.3.13]{miyake}, the 
assumed analytic properties of $\Lambda_f(s,\chi)$ and
$\Lambda_g(s,\overline{\chi})$ described in (ii), together with the
functional equation \eqref{eq:femult} for all primitive $\chi\pmod*{q}$,
imply the equality
$$
f\bigg|\Bigl[\gamma_{q,a}-\overline{\xi(q)}\Bigr]
\bpm 1 & \frac{a}{q} \\ 0 & 1\ebpm
=f\bigg|\Bigl[\gamma_{q,b}-\overline{\xi(q)}\Bigr]
\bpm 1 & \frac{b}{q} \\ 0 & 1\ebpm
$$
for any integers $a,b$ coprime to $q$.
By Lemma~\ref{lem:cq}, it follows that
$f|\gamma_{q,a}=\overline{\xi(q)}f$ for every $a$ coprime to $q$.
By the definition of $Q$, we thus have
$f|\gamma=\xi(\gamma)f$ for every $\gamma\in H_q\supseteq\Gamma_1(N)$.
Applying Theorem~\ref{thm:finiteindex} with $H=\Gamma_1(N)$,
we conclude that $f\in M_k(\Gamma_0(N),\xi)$.
\end{proof}

\section{Generating $\Gamma_1(N)$}\label{sec:generating}
In this section, we consider the question of when the elements of
$\Gamma_0(N)$ with a fixed upper-left entry generate a
subgroup containing $\Gamma_1(N)$. By the proof of
Theorem~\ref{thm:singleq}, any such upper-left entry gives sufficient
conditions to imply modularity using twists of a single modulus.

For any $q\in\N$ coprime to $N$, let $H_q$ denote the subgroup of
$\Gamma_0(N)$ generated by the matrices
$$
\left\{\begin{pmatrix}A&B\\C&D\end{pmatrix}\in\Gamma_0(N):A=q\right\}.
$$
\begin{conjecture}
There exists $q_0=q_0(N)$ such that
$H_q\supseteq\Gamma_1(N)$ for every $q\ge q_0$ coprime to $N$.
\end{conjecture}

\begin{theorem}\label{thm:almostall}
$H_q\supseteq\Gamma_1(N)$ holds for almost all $q\in\N$ coprime to $N$
and for almost all primes $q\nmid N$, i.e.\
\begin{equation}\label{eq:almostall1}
\#\{q\in\N:(q,N)=1,\;H_q\supseteq\Gamma_1(N),\;q\le x\}
=\bigl(\tfrac{\varphi(N)}{N}+o(1)\bigr)x
\end{equation}
and
\begin{equation}\label{eq:almostall2}
\#\{q\text{ prime}:q\nmid N,\;H_q\supseteq\Gamma_1(N),\;q\le x\}
=(1+o(1))\pi(x)
\end{equation}
as $x\to\infty$.
\end{theorem}
\begin{proof}
For $q\in\N$ coprime to $N$, set
$$
\Gamma_q=
\left\{\begin{pmatrix}A&B\\C&D\end{pmatrix}\in\Gamma_0(N):
A\equiv q^n\pmod*{N}
\text{ for some }n\in\N\right\}.
$$
Then $\Gamma_q$ is a group satisfying
$\Gamma_1(N)\cup H_q\subseteq\Gamma_q\subseteq\Gamma_0(N)$,
and we have
$$
H_q\supseteq\Gamma_1(N)\Longleftrightarrow
H_q=\Gamma_q.
$$

Consider a fixed $q_0\in\N$ coprime to $N$, and let
$\bar{q}_0$ be a multiplicative inverse of $q_0\pmod*{N}$.
Then, for any $q\equiv q_0\pmod*{N}$,
$$
T=\begin{pmatrix}q&1\\q(N+\bar{q}_0)-1&\bar{q}_0+N\end{pmatrix}
\begin{pmatrix}q&1\\q\bar{q}_0-1&\bar{q}_0\end{pmatrix}^{-1},
$$
and
$$
W=\begin{pmatrix}q&1\\q\bar{q}_0-1&\bar{q}_0\end{pmatrix}^{-1}
\begin{pmatrix}q&q+1\\q\bar{q}_0-1&(q+1)\bar{q}_0-1\end{pmatrix}.
$$
so that $H_q$ and $\Gamma_q=\Gamma_{q_0}$ contain $\langle T,W\rangle$.

Let
$$
\{T,W\}\cup\left\{\gamma_i=
\begin{pmatrix}A_i&B_i\\NC_i&D_i\end{pmatrix}:1\le i\le h
\right\}
$$
be a fixed generating set for $\Gamma_{q_0}$, with
$\gamma_1=\sm q_0&1\\q_0\bar{q}_0-1&\bar{q}_0\esm$.
For $i\ge2$, replacing $\gamma_i$ by $\gamma_1^{n_i}\gamma_i$ for a
suitable $n_i$, we may assume that $A_i\equiv q_0\pmod*{N}$. Also,
we may assume that $A_i\ne0$, since otherwise $N=1$ and $\gamma_i$
is contained in $\langle T,W\rangle$.

Next, we modify $\gamma_1,\ldots,\gamma_h$ by multiplying by powers of
$T$ and $W$. First, multiplying by $W^{m_i}$ on the left leaves $A_i$
unchanged and replaces $C_i$ by $C_i+m_iA_i$. Hence, by Dirichlet's
theorem, we may take $C_1,\ldots,C_h$ to be distinct primes not dividing
$N$.  Second, by the Chinese remainder theorem, we can choose $q_1\in\N$
satisfying $q_1\equiv q_0\pmod*{N}$ and $q_1\equiv A_i\pmod*{C_i}$ for
every $i$. Multiplying on the left by $T^{(q_1-A_i)/(NC_i)}$ replaces
each $A_i$ by $q_1$.

Now, let $q\in\N$ with $q\equiv q_0\pmod*{N}$.
Suppose that the divisors of $q-q_1$ represent all invertible
residue classes modulo $Nq_1$, i.e.\
\begin{equation}\label{eq:divisors}
\{d+Nq_1\Z:d\in\N,\;d\mid(q-q_1)\}\supseteq(\Z/Nq_1\Z)^\times.
\end{equation}
For $i=1,\ldots,h$, let $d_i$ be a divisor of $q-q_1$ satisfying
$d_i\equiv C_i\pmod*{Nq_1}$. Then $(d_i,N)=1$, so $Nd_i\mid(q-q_1)$.
Hence,
$$
T^{\frac{q-q_1}{Nd_i}}W^{\frac{d_i-C_i}{q_1}}
\begin{pmatrix}q_1\\NC_i\end{pmatrix}
=\begin{pmatrix}q\\Nd_i\end{pmatrix},
$$
so that $\gamma_i$ is contained in $H_q$. Therefore
$H_q=\Gamma_{q_0}=\Gamma_q$.

Erd\H{o}s \cite{erdos} showed that almost all $q\in\N$ satisfy
\eqref{eq:divisors}. Therefore, the set of $q\in\N$
such that $q\equiv q_0\pmod*{N}$ and $H_q=\Gamma_q$ has density
$1/N$. Letting $q_0$ run through a set of representatives
for the invertible residue classes mod $N$ yields \eqref{eq:almostall1}.
For the prime case, we similarly apply Lemma~\ref{lem:divisors} with
$(p_0,q)=(q_1,Nq_1)$ to see that almost all $q\nmid N$ satisfy
\eqref{eq:divisors}, and this leads to \eqref{eq:almostall2}.
\end{proof}

\begin{theorem}\label{thm:explicitq}
For each $N$ in the following table,
$H_q\supseteq\Gamma_1(N)$ holds for $q\in\N$ with $(q,N)=1$ and for 
primes $q\nmid N$ in the indicated intervals.
\begin{center}
\begin{tabular}{rrr|rrr}
$N$  & $ (q,N)=1$       & prime $q\nmid N$  &  $N$  & $(q,N) = 1$       & prime $q \nmid N$ \\ \hline
$5$  & $ (44,10^9)$     & $(0,10^9)$        &  $14$ & $ (55, 10^9)$     & $(43, 10^9)$      \\
$6$  & $ (1,10^9)$      & $(0,10^9)$        &  $15$ & $ (91, 10^9)$     & $(31, 10^9)$      \\
$7$  & $ (20,10^9)$     & $(0,10^9)$        &  $16$ & $ (63, 10^9)$     & $(47, 10^9)$      \\
$8$  & $ (15,10^9)$     & $(7,10^9)$        &  $17$ & $ (390, 10^5)$    & $(101,10^9)$     \\
$9$  & $ (136,10^9)$    & $(2,10^9)$        &  $18$ & $ (55, 10^9)$     & $(53, 10^9)$      \\
$10$ & $ (39,10^9)$     & $(11,10^9)$       &  $19$ & $ (360, 10^5)$    & $(37, 10^9)$      \\
$11$ & $ (84, 10^9)$    & $(2, 10^9)$       &  $20$ & $ (119, 10^5)$    & $(79, 10^9)$      \\
$12$ & $ (35, 10^9)$    & $(23, 10^9)$      &  $21$ & $ (230, 10^5)$    & $(83, 10^9)$      \\
$13$ & $ (168, 10^9)$   & $(5, 10^9)$       &  $22$ & $ (175, 10^5)$    & $(43, 10^9)$      \\
\end{tabular}
\end{center}
\end{theorem}
\begin{proof}
We applied two strategies to verify the statement computationally.
First, we used Lemma~\ref{lem:height} and Corollary~\ref{cor:TWfree}
to compute a list $L$ of all elements of $\langle T,W\rangle$ of height
up to some bound chosen by trial and error (e.g.\ for $N=13$ we chose
the bound $5500$, which yielded $290841$ words in $T,W$). We then
used Sage \cite{sage} to compute a generating set $\{g_1,\ldots,g_h\}$
for $\Gamma_1(N)$, and for each generator we computed every word of
the form $w_1g_i^{\pm1}w_2$, for $w_1,w_2\in L$. Combining this with
Lemma~\ref{lem:Gamma1gen} and a simple sieve, we obtained sufficient
conditions to establish the claim for the vast majority of $q$.

For the relatively small number of values of $q$ remaining, we computed the
expansions of every
element $\gamma_{q,a}$ for $1\le a\le q$ in terms of the generators
$S=\sm&-1\\1&\esm$ and $T=\sm1&1\\&1\esm$ of $\SL_2(\Z)$, and presented
$\SL_2(\Z)\cong\langle S,T:S^4=S^2(ST)^3=1\rangle$ as an abstract group to
GAP \cite{gap}. We then used GAP's implementation of the Todd--Coxeter
algorithm \cite{todd-cox} to attempt to compute the index $[\SL_2(\Z):H_q]$.
When this terminated with a number equal to the expected index
$[\SL_2(\Z):\Gamma_q]$, we obtained the claim for $q$.

The first strategy tends to work better at finding prime values of $q$,
which explains the discrepancy in the sizes of the intervals for larger
values of $N$, where there are eventually too many exceptions to test by
the second method in a reasonable amount of time.

For some $q$ (those for which the Todd--Coxeter algorithm appeared not
to terminate), our results were inconclusive, though we expect that
$H_q\not\supseteq\Gamma_1(N)$ in those cases. In a very small number of
cases, $H_q$ has finite index in $\SL_2(\Z)$ but is not
the full group $\Gamma_q$.
\end{proof}

\section{Lemmas}
\begin{lemma}\label{lem:cq}
Let $q\in\N$ with $(q,N)=1$.
The assumptions of Conjecture~\ref{main} imply the relation
\begin{equation}\label{e:gz_cqgen}
	\sum_{\substack{a\pmod*{q}\\(a,q)=1}} f\bigg|
	\Bigl[\gamma_{q,a} - \overline{\xi(q)}\Bigr] \bpm 1 & \frac{a}{q} \\ 0 & 1\ebpm
	=0, 
\end{equation}
where $\gamma_{q,a}$ is any element of $\Gamma_0(N)$ with top row $\sm q&-a\esm$.
\end{lemma}
\begin{proof}
From Hecke \cite[Theorem~4.3.5]{miyake}
we know that the functional equation in Conjecture~\ref{main}
is equivalent to the equation
\begin{equation}
\sum_{n=1}^{\infty}f_n c_q(n) e^{2\pi i nz} = (-1)^k \xi(q) (Nq^2)^{-\frac{k}{2}} z^{-k} \sum_{n=1}^{\infty} g_n c_q(n) e^{2 \pi i \frac{-n}{Nq^2z}}.
\label{eq:modularrelation}
\end{equation}
In particular we find for $q=1$, that $f|\sm 0 & -N^{-\frac{1}{2}} \\ N^{\frac{1}{2}} & 0 \esm=g$, where $g(z)=\sum_{n=1}^{\infty} g_n e^{2 \pi i n z}$. Now we shall note that \eqref{eq:modularrelation} may be rewritten as
\begin{equation*}
\sum_{\substack{a\pmod*{q}\\(a,q)=1}} f\bigg| \bpm 1 & \frac{a}{q} \\ 0 & 1\ebpm
	=\xi(q)\sum_{\substack{c\pmod*{q}\\(c,q)=1}}
	g\bigg| \bpm -N^{\frac{1}{2}}c & N^{-\frac{1}{2}}q^{-1} \\ -N^{\frac{1}{2}}q & 0\ebpm.
\end{equation*}
Combining this with the matrix identity
\begin{equation*}
\bpm 0 & -N^{-\frac{1}{2}} \\ N^{\frac{1}{2}} & 0 \ebpm \bpm
-N^{\frac{1}{2}}c & N^{-\frac{1}{2}}q^{-1} \\ -N^{\frac{1}{2}}q & 0\ebpm
= \bpm q & 0 \\ -Nc & q^{-1}\ebpm = \bpm q & -a \\ -Nc & s\ebpm \bpm 1 &
\frac{a}{q} \\ 0 & 1\ebpm,
\end{equation*}
where $a=a(c)$ is chosen so that $Nca\equiv-1\pmod*{q}$ and $s=(Nac+1)/q$,
we derive
\begin{equation*}
\sum_{\substack{a\pmod*{q}\\(a,q)=1}}
f\bigg| \bpm 1 & \frac{a}{q} \\ 0 & 1\ebpm
	=\xi (q) \sum_{\substack{c\pmod*{q}\\(c,q)=1}} f\bigg|
	\bpm q & -a \\ -Nc & s\ebpm  \bpm 1 & \frac{a}{q} \\ 0 & 1\ebpm.
\end{equation*}
Here the summation over $c$ may be replaced by the summation over
$a\pmod*{q}, (a,q)=1$, by choosing appropriate representatives, thereby proving the lemma.
\end{proof}

\begin{lemma}\label{lem:weil}
Suppose that $h:\HH\to\C$ is a holomorphic function,
$M\in\SL_2(\R)$ is elliptic of infinite order,
and $\zeta\in\C^\times$ is a root of unity such that
$h|M=\zeta h$. Then $h=0$.
\end{lemma}
\begin{proof}
This is an extension of Weil's Lemma \cite[Lemma~1.5.1]{bump}, which
is the special case $\zeta=1$. It can be proven by the same method or,
alternatively, derived as a consequence, as follows.
Suppose that $\zeta$ has order $n$, and let $M=\sm a&b\\c&d\esm$.
Then we have
$$
(cz+d)^{-kn}h\!\left(\frac{az+b}{cz+d}\right)^n
=\bigl((h|M)(z)\bigr)^n=h(z)^n.
$$
Applying Weil's Lemma to $h^n$ (and the weight-$kn$ slash operator),
we conclude that $h^n=0$, whence $h=0$.
\end{proof}

\begin{lemma}\label{lem:conrey-farmer}
Assume the hypotheses of Conjecture~\ref{main}, and
suppose that $N=qs-1$, where $q,s\in\{3,4,6\}$. Then
$f|\gamma_{q,1}=\overline{\xi(q)}f$.
\end{lemma}
\begin{proof}
Note that $\varphi(q)=\varphi(s)=2$, and
we have $\gamma_{q,\pm1}=\gamma_{s,\mp1}^{-1}=\sm q&\mp1\\\mp N&s\esm$.
Hence, applying Lemma~\ref{lem:cq} to both $q$ and $s$, we obtain
\begin{align*}
f\bigl|\bigl[\gamma_{q,1}-\overline{\xi(q)}\bigr]
&=-f\biggl|\bigl[\gamma_{q,-1}-\overline{\xi(q)}\bigr]\bpm 1&-2/q\\&1\ebpm
=\xi(s)f\biggl|\bigl[\gamma_{s,1}-\overline{\xi(s)}\bigr]
\gamma_{s,1}^{-1}\bpm 1&-2/q\\&1\ebpm\\
&=-\xi(s)f\biggl|\bigl[\gamma_{s,-1}-\overline{\xi(s)}\bigr]
\bpm 1&-2/s\\&1\ebpm\gamma_{s,1}^{-1}\bpm 1&-2/q\\&1\ebpm\\
&=f\biggl|\bigl[\gamma_{q,1}-\overline{\xi(q)}\bigr]
\gamma_{q,1}^{-1}\bpm 1&-2/s\\&1\ebpm\gamma_{s,1}^{-1}\bpm 1&-2/q\\&1\ebpm.
\end{align*}
Writing $M=\gamma_{q,1}^{-1}\sm 1&-2/s\\&1\esm
\gamma_{s,1}^{-1}\sm 1&-2/q\\&1\esm
=\sm1&-2/q\\2q-2/s&-3+4/(qs)\esm$,
we thus have
$$
f\bigl|\bigl[\gamma_{q,1}-\overline{\xi(q)}\bigr][I-M]=0.
$$
Note that $|\tr{M}|<2$ and $\tr{M}\notin\Z$, so $M$ is elliptic of
infinite order. Applying Lemma~\ref{lem:weil} to
$h=f|[\gamma_{q,1}-\overline{\xi(q)}]$, we obtain
$f|\gamma_{q,1}=\overline{\xi(q)}f$.
\end{proof}

\begin{lemma}\label{lem:1qinv}
Assume the hypotheses of Conjecture~\ref{main}, and suppose there exist
$\gamma=\sm A&B\\C&D\esm\in\Gamma_0(N)$, $\alpha\in\Q$ and a root of unity
$\zeta\in\C^\times$ such that $C\alpha\notin\Z$, $|A+D+C\alpha|<2$, and
$$
f\bigl|\bigl[\gamma^{-1}-\xi(A)\bigr]
=\zeta f\biggl|[\gamma-\xi(D)]\bpm 1&\alpha\\&1\ebpm.
$$
Then $f|\gamma=\xi(D)f$.
\end{lemma}
\begin{proof}
We have
$$
-\xi(D)\zeta f|[\gamma-\xi(D)]
=-\xi(D)f\biggl|\bigl[\gamma^{-1}-\xi(A)\bigr]
\bpm 1&-\alpha\\&1\ebpm
=f\biggl|[\gamma-\xi(D)]
\gamma^{-1}\bpm 1&-\alpha\\&1\ebpm.
$$
Note that
$\tr\bigl(\gamma^{-1}\sm1&-\alpha\\&1\esm\bigr)=A+D+C\alpha$.
By hypothesis this is non-integral and has modulus less than $2$,
so $\gamma^{-1}\sm1&-\alpha\\&1\esm$ is elliptic of infinite order.
Applying Lemma~\ref{lem:weil}, we obtain $f|\gamma=\xi(D)f$.
\end{proof}

\begin{lemma}\label{lem:key}
Let $h,n,m\in\N$, and let
$q_1,\ldots,q_h$ be distinct primes with $q_j\nmid mn$ for all $j$.
For every $j$, let $s_{i,j} \subseteq \{1,\ldots,q_j-1\}$, with $s_{i_1,j}\cap s_{i_2,j}=\emptyset$
for all $i_1\neq i_2$ (we do not assume that $s_{i,j}\neq \emptyset$).
Let $S_{i,j} = \sum_{a\in s_{i,j}} e\big(\frac{na}{mq_j}\big)$.
Suppose that $s_{i,i}\neq \emptyset$ for every $i$.
Then $\det\bigl([S_{i,j}]_{1\le i,j\le h}\bigr)\neq 0$.
\end{lemma}
\begin{proof}
Replacing $(m,n)$ by $(m/\gcd(m,n),n/\gcd(m,n))$ if necessary, we may
assume without loss of generality that $(m,n)=1$.
We prove the claim by induction on $h$.

Suppose first that $h=1$. Each $e\big(\frac{na}{mq_1}\big)$ is the $a$th power of $e\big(\frac{n}{mq_1}\big) =:
\zeta_{mq_1}$, which is a primitive $mq_1$th root of unity. By
hypothesis $s_{1,1}$ is not empty, so
$S_{1,1}$ is the value at $\zeta_{mq_1}$ of a nonconstant
polynomial $P\in\Q[x]$. Note that
$P(x)=xQ(x)$ for some nonzero $Q\in\Q[x]$ (since $s_{1,1}\subseteq \{1,\ldots,q_1-1\}$), and that the degree of
$Q$ is at most $q_1-2$. The degree of the extension $\Q(\zeta_{mq_1})/\Q$ is
$\eul(mq_1) = \eul(m)\eul(q_1) \geq \eul(q_1) = q_1-1$.
Hence $S_{1,1}=P(\zeta_{mq_1})=\zeta_{mq_1}Q(\zeta_{mq_1})\neq
0$. This concludes the proof for $h=1$.

Suppose $h\geq 2$ and expand $\det[S_{i,j}]$ with respect to the first line.
We get an expression of the form $P(\zeta_{mq_1})$ for some polynomial
$P\in\Q(\zeta_{mq_2},\ldots,\zeta_{mq_h})[x]$.
We claim that $P$ is not constant. To see this, let $a\in s_{1,1}$
(such $a$ exists because $s_{1,1}\neq\emptyset$). Then
$a\notin s_{i,1}$ for any $i\neq 1$, since
$s_{i_1,1}\cap s_{i_2,1}=\emptyset$ for $i_1\neq i_2$. Thus, the coefficient of $x^a$ in $P(x)$
is the determinant of the cofactor matrix for $S_{1,1}$.
This determinant satisfies all hypotheses of the
lemma for $h-1$ and primes $q_2,\ldots,q_h$; hence it is nonzero by
the inductive hypothesis.

Note that $P(x)=xQ(x)$ for some nonzero
$Q\in\Q(\zeta_{mq_2},\ldots,\zeta_{mq_h})[x]$
(since each $s_{i,1}\subseteq\{1,\ldots,q_1-1\}$), and that the degree of $Q$ is $\leq q_1-2$. By coprimality assumptions, the degree
of the extension
$\Q(\zeta_{mq_1},\ldots,\zeta_{mq_h})/\Q(\zeta_{mq_2},\ldots,\zeta_{mq_h})$
is $\eul(mq_1q_2\cdots q_h)/ \eul(mq_2\cdots q_h)= \eul(q_1) = q_1-1$. Hence
$Q(\zeta_{mq_1})\neq0$. Thus $\det[S_{i,j}] = P(\zeta_{mq_1})= \zeta_{mq_1}Q(\zeta_{mq_1}) \neq 0$.
\end{proof}

\begin{lemma}\label{lem:nonvanishing}
Assume the hypotheses of Conjecture~\ref{main}, and suppose that $f$ is
not identically $0$. Then for any prime $q\nmid N$, there exists
$n\in\N$ such that $q\mid n$ and $\{f_n,g_n\}\ne\{0\}$.
\end{lemma}
\begin{proof}
Suppose that the conclusion is false for some prime $q\nmid N$, so that
$f_n=g_n=0$ for every $n$ divisible by $q$.
Then we have
$f_nc_q(n)=-f_n$ and $g_nc_q(n)=-g_n$ for every $n$, so that
$$
-1=\frac{\Lambda_f(s,c_q)}{\Lambda_f(s,c_1)}
=\frac{\Lambda_g(1-s,c_q)}{\Lambda_g(1-s,c_1)}.
$$
On the other hand, \eqref{e:fe_cq} applied to $c_1$ and $c_q$ shows that
$$
\frac{\Lambda_f(s,c_q)}{\Lambda_f(s,c_1)}
=\xi(q)q^{1-2s}\frac{\Lambda_g(1-s,c_q)}{\Lambda_g(1-s,c_1)},
$$
so $\xi(q)q^{1-2s}=1$. Since $q>1$, this is a contradiction.
\end{proof}

\begin{lemma}\label{lem:log-gen}
Let $N$ be a prime, and for each $r\in\Z$ with $2\le|r|<\frac12N$,
let $\gamma_{r,1}\in\Gamma_0(N)$ be a matrix with top row $\sm r&-1\esm$.
Then any matrix $\sm A & B \\ CN & D \esm \in \Gamma_0(N)$ may be
written in the form
$\pm\tau_1\tau_2\cdots\tau_l$
with
$\tau_i\in\{T,T^{-1},W,W^{-1},\gamma_{r,1}^{-1}:2\le|r|<\frac12N\}$
for each $i=1,\ldots,l$, in such a way that
$$
\#\{i:\tau_i\in\{\gamma_{r,1}^{-1}\}\}\le\log_2(|A|).
$$
\end{lemma}
\begin{proof}
If $C=0$ then $\sm A & B \\ CN & D \esm = \pm T^{\alpha}$ for some
choice of sign and $\alpha\in\Z$. In the general case we may multiply
on the left by a power of $T$ to replace $A$ by any integer $A'$
such that $A'\equiv A\pmod*{CN}$. Choosing $A'$ such that
$|A'|\le\frac12|CN|$, we also have $|A'|\le|A|$. Similarly we
may multiply on the left by $W$ and replace $C$ by any integer
$C'\equiv C\pmod*{A'}$ with $|C'|\le\frac12|A'|$.

Repeating this process will either lead to $C=0$ or will eventually
stagnate. Thus we may assume now that $|A|\le\frac12|CN|$ and
$0<|C|\le\frac12|A|$. In particular, this implies that $N\ge4$, so $N$
is an odd prime.  Let $r$ be the nearest integer to the fraction $CN/A$
(note that $A\ne0$ since $(A,N)=1$), rounded toward $0$ in the case of
a tie. We have $2\le|CN/A|\le\frac12N$, and thus $2\le|r|<\frac12N$.
Multiplying on the left by $\gamma_{r,1}$, the new top-left corner is
$rA-CN=A(r-\frac{CN}{A})$, which does not exceed $\frac12|A|$ in absolute
value. Thus, by repeating this process we eventually end up in the case
$C=0$, having used at most $\log_2(|A|)$ matrices $\gamma_{r,1}$.
\end{proof}

\begin{lemma}\label{lem:diagonal}
Let $A$ be an $n\times n$ matrix over a ring, with non-zero rows.
Then there exists $m\in\{1,\ldots,n\}$ and $n\times n$ permutation
matrices $P$ and $Q$ such that $PAQ$ takes the block form
$\left(\begin{array}{c|c}
\hat{A}&0\\
\hline
C&D\end{array}\right)$,
where $\hat{A}$ is of size $m\times m$ and has non-zero diagonal
entries.
\end{lemma}
\begin{proof}
Denote the entries of $A$ by $a_{ij}$.
For any $S\subseteq\{1,\ldots,n\}$, define
$$
m_S=\#\{j:a_{ij}\ne0\text{ for some }i\in S\}.
$$
Note that for $S=\{1,\ldots,n\}$ we have $m_S\le\#S$. Hence, there is
a minimal non-empty set $R\subseteq\{1,\ldots,n\}$ satisfying $m_R\le\#R$.
Since $A$ has non-zero rows, we have $m_S>0$ whenever $S\ne\emptyset$.
From this and the minimality of $R$ it follows that $m_R=\#R$.
Moreover, for any $S\subseteq R$ we have $m_S\ge\#S$.

By Hall's marriage theorem \cite{Hall1935}, it follows that there is a subset
$C\subseteq\{1,\ldots,n\}$ and a bijection $i:C\to R$ such that
$a_{i(j)j}\ne0$ for every $j\in C$. Writing $m=\#C=\#R$ and replacing $A$
by $PAQ$ for appropriate permutation matrices $P$ and $Q$, we may assume
that $C=R=\{1,\ldots,m\}$ and $i(j)=j$.
The block form of $A$ then follows from the definition of $m_S$.
\end{proof}

\begin{lemma}\label{lem:divisors}
Given $p_0,a,q\in\Z$ with $p_0\ne0$ and $(a,q)=1$, define
$$
P(p_0;a,q)=\{p\text{ prime}:
\exists d\in\N\text{ such that }d\equiv a\pmod*{q}
\text{ and }p\equiv p_0\pmod*{d}\}
$$
and
$$
P(p_0;q)=\bigcap_{\substack{1\le a\le q\\(a,q)=1}}P(p_0;a,q).
$$
Then
$$
\#\{p\in P(p_0;q):p\le x\}=(1+o(1))\pi(x)
\quad\text{as }x\to\infty.
$$
\end{lemma}
\begin{proof}
This is proven for $p_0=1$ in~\cite{Hall1973}, uniformly for
$q\le2^{(1-\varepsilon)\log\log{x}}$. One can generalize the proof to
all $p_0\neq0$, and if one is not concerned with the uniformity in $q$
a simpler proof suffices. For completeness we give the argument here.

For a character $\chi$ modulo $q$ and $a\in\Z$ with $(a,q)=1$ let
$$
d_\chi(n):=\sum_{d\mid n}\chi(d),\qquad
d(n;a):=\sum_{\substack{d\mid n\\d\equiv a\pmod*{q}}}1,
$$
so that we have
\begin{align}\label{charexp}
d(n;a)=\frac1{\varphi(q)}\sum_{\chi\pmod*{q}}\overline\chi(a)d_\chi(n).
\end{align}
Then, it suffices to prove that for almost all primes $p$, $d(p-p_0;a)>0$
for all $a\pmod*{q}$ with $(a,q)=1$.

As in \cite{Hall1973} we start by observing that if $p',n$ are coprime
with $p'$ prime, then by multiplicativity and the Cauchy--Schwarz
inequality one has
$$
\Big(d(np';a)-\frac{d_{\chi_0}(np')}{\varphi(q)}\Big)^2
\leq 16\sum_{\substack{b\pmod*{q}\\(b,q)=1}}
\Big(d(n;b)-\frac{d_{\chi_0}(n)}{\varphi(q)}\Big)^2,
$$
where $\chi_0$ is the trivial character modulo $q$.
Denoting by $\omega(n)$ the number of distinct prime factors of $n$,
Halberstam \cite{Halberstam} proved that $\omega(p-p_0)$ has normal order
$\log \log p$. Thus, $\omega(p-p_0)\leq 2\log\log p$ for almost all $p\leq
x$ and so, in particular, $p-p_0$ almost always has a prime factor $p'$
greater than $r(x):=x^{\frac1{4\log \log x}}$ as $x\to\infty$. Also for
almost all such $p$ we have $(p',(p-p_0)/p')=1$ since only $o(\pi(x))$
integers $\leq x$ have such a large repeated prime factor. Denoting
by $\sum'$ the restriction of the sum to primes with such properties, we then have
\begin{align*}
\sumprime_{p-p_0\leq x}\Big(d(p-p_0;a)
-\frac{d_{\chi_0}(p-p_0)}{\varphi(q)}\Big)^2
&\leq 16\sum_{\substack{b\pmod*{q}\\(b,q)=1}}
\sum_{\substack{p-p_0=np'\leq x,\\ p,p'\text{ primes},\\ p'\geq r(x),\,(n,p')=1}}\Big(d(n;b)-\frac{d_{\chi_0}(n)}{\varphi(q)}\Big)^2\\
&\ll\max_{\substack{b\pmod*{q}\\(b,q)=1}}\sum_{n\leq \frac{x}{r(x)}}
\Big(d(n;b)-\frac{d_{\chi_0}(n)}{\varphi(q)}\Big)^2\sum_{\substack{p-p_0=np'\leq x,\\ p,p'\text{ primes}}}1,
\end{align*}
where all the implicit constants here and below are allowed to depend on $q,p_0$.
By~\cite[Ch.~II~Satz~4.2]{Prachar} (cf.\ Satz~4.6 for the case $p_0=1$),
with $(a_1,b_1,a_2,b_2)=(1,0,n,p_0)$, the inner sum is
$O(\frac{x}{\varphi(n)\log^2( x/n)})
=O(\frac{x(\log \log x)^2}{\varphi(n)\log^2 x})$ since $n\leq x/r(x)$.
Thus, using also~\eqref{charexp} the above is
\begin{align}\label{alfi}
\ll \frac{x(\log \log x)^2}{\log^2x}\max_{\substack{b\pmod*{q}\\(b,q)=1}}
\sum_{\chi_0\neq\chi_1,\chi_2\pmod*{q}}\frac{\chi_1(b)\overline\chi_2(b)}{\varphi(q)^2}\sum_{n\leq \frac{x}{r(x)}}\frac{d_{\chi_1}(n)d_{\chi_2}(n)}{\varphi(n)}.
\end{align}
An easy exercise shows that for $\Re(s)>1$,
$$
\sum_{n\geq1}\frac{d_{\chi_1}(n)d_{\chi_2}(n)}{\varphi(n)n^s}=L(1+s,\chi_0)L(1+s,\chi_1)L(1+s,\chi_2)L(1+s,\chi_1\chi_2)R(s)
$$
where $R(s)$ is an Euler product which is convergent and uniformly bounded on $\Re(s)\geq-\frac14$. It follows that the inner sum in~\eqref{alfi} is $O(\log^2 x)$. Thus we find
$$
\sumprime_{p-p_0\leq x}\Big(d(p-p_0;a)-\frac{d_{\chi_0}(p-p_0)}{\varphi(q)}\Big)^2\ll x(\log \log x)^2
$$
and so we deduce that for $\eps>0$ we must have 
$$
d(p-p_0;a)-\frac{d_{\chi_0}(p-p_0)}{\varphi(q)}\ll_{\eps} (\log x)^{\frac12+\eps}
$$
for almost all $p\leq x$.
Finally,  for almost all primes $p\leq x$ we have $\omega(p-p_0)\ge(1-\eps)\log \log x $ and so
$$
d_{\chi_0}(p-p_0)\geq 2^{\omega(p-p_0)-\omega(q)}\gg_{\eps} (\log x)^{\log 2-\eps}.
$$
Since $\log 2>1/2$ we deduce that for almost all primes $p\leq x$ we have
$$
d(p-p_0;a)\gg_{\eps} (\log x)^{\log 2-\eps},
$$
as desired.
\end{proof}

\begin{lemma}\label{lem:Lambda_chi}
Let $f\in M_k(\Gamma_0(N),\xi)$, and define $g$ by \eqref{eq:gdef}.
Let $f_n$ and $g_n$ denote the Fourier coefficients
of $f$ and $g$, respectively, and for any character $\chi$ of modulus
$q$ coprime to $N$, define $\Lambda_f(s,c_\chi)$ and
$\Lambda_g(s,c_{\overline\chi})$ as in \eqref{e:Lambda}.
Then $\Lambda_f(s,c_\chi)$ and
$\Lambda_g(s,c_{\overline\chi})$ continue to entire functions, apart
from at most simple poles at $s=\frac{1\pm k}2$, and satisfy the
functional equation \eqref{eq:fe_cchi}.
\end{lemma}
\begin{proof}
Define
\begin{equation}\label{e:fchi}
	f_{\chi}(z) := 
	\sum_{\substack{a\pmod*{q}\\(a,q)=1}} \chi(a) f\bigg| \bpm 1 & \frac{a}{q}\\ & 1\ebpm 
	=
	\sum_{n=0}^\infty f_n c_\chi(n) e(nz),
\end{equation}
and similarly for $g_{\overline\chi}$.
Then
\begin{equation}
	f_\chi\left|\bpm&-1\\Nq^2\ebpm\right.
	=
	\sum_{\substack{u\pmod*{q}\\(u,q)=1}}\chi(u) f\left|
	\bpm 1 & \frac{u}{q} \\ & 1\ebpm
	\bpm&-1\\Nq^2\ebpm\right..
\end{equation}
Since 
\begin{equation}
	q^{-1}
	\bpm&-1\\N\ebpm^{-1}
	\bpm 1 & \frac{u}{q}\\ & 1\ebpm
	\bpm&-1\\Nq^2\ebpm
	\bpm 1 & -\frac{v}{q}\\ & 1\ebpm 
	=
	\bpm q & -v\\ -uN & \frac{1+uvN}{q}\ebpm \in \Gamma_0(N), 
\end{equation}
provided that $uvN\equiv -1\pmod*{q}$, 
we have
\begin{multline}
	f_\chi\left|\bpm&-1\\Nq^2\ebpm\right.
	=
	\xi(q)
	\sum_{\substack{u\pmod*{q}\\ uvN \equiv -1\pmod*{q}}}
	\chi(u)  g\bigg|\bpm 1 & \frac{v}{q}\\ & 1\ebpm
	\\
	=
	\xi(q)\overline{\chi(-N)}
	\sum_{\substack{u\pmod*{q}\\ uvN \equiv -1\pmod*{q}}}
	\overline{\chi(v)}  g\bigg|\bpm 1 & \frac{v}{q}\\ & 1\ebpm
	=
	\xi(q)\overline{\chi(-N)}
	g_{\overline{\chi}}.
\end{multline}
The conclusion now follows by Hecke's argument
\cite[Theorem~4.3.5]{miyake}.

\end{proof}

\begin{lemma}
Let $\chi\pmod*{q}$ be a Dirichlet character induced by the primitive
character $\chi_*\pmod*{q_*}$.
Define $q_0 = \prod_{p\mid q, p\nmid q_*}p$ and $q_2 = \frac{q}{q_* q_0}$. 
Then $c_\chi(n)=0$ if $q_2\nmid n$, and  
\begin{equation}\label{e:cchi_decomp}
	c_\chi(nq_2)=
	q_2\chi_*(q_0)c_{\chi_*}(n)c_{q_0}(n)
	=
	q_2\chi_*(q_0)
	\tau(\chi_*)
	\mu(q_0)
	\overline{\chi_*(n)}
	\mu(\gcd(q_0, n))
	\varphi(\gcd(q_0, n)).
\end{equation}
\end{lemma}
\begin{proof}
By \cite[\S9.2, Theorem 12]{MV}, if $q_*\mid\frac{q}{\gcd(q, n)}$ then 
$$
	c_\chi(n) = 
	\overline{\chi_*\!\left(\frac{n}{\gcd(q, n)}\right)} \chi_*\!\left(\frac{q}{\gcd(q, n) q_*}\right) 
	\mu\!\left(\frac{q}{\gcd(q, n) q_*}\right) \frac{\varphi(q)}{\varphi\!\left(\frac{q}{\gcd(q, n)}\right)} \tau(\chi_*),
$$
and $c_\chi(n)=0$ otherwise.
Since 
$\chi_*\!\left(\frac{q}{\gcd(q, n) q_*}\right) =
\chi_*\!\left(\frac{q_0q_2}{\gcd(q, n)}\right)=0$
unless $q_2\mid n$, 
we get $c_\chi(n)=0$ if $q_*\nmid \frac{q}{\gcd(q, n)}$ or $q_2\nmid n$. 

For an integer $n$, we get
\begin{multline*}
	c_\chi(nq_2)=
	\overline{\chi_*\!\left(\frac{n}{\gcd(q_0, n)}\right)}
	\chi_*\!\left(\frac{q_0}{\gcd(q_0, n)}\right)
	\mu\!\left(\frac{q_0}{\gcd(q_0, n)} \right)
	\frac{\varphi(q)}{\varphi\!\left(q_* \frac{q_0}{\gcd(q_0, n)}\right)}
	\tau(\chi_*)
	\\
	=
	\overline{\chi_*(n)}
	\chi_*(q_0)
	\tau(\chi_*)
	\mu(q_0) 
	\frac{\varphi(q)}{\varphi(q_* q_0)}
	\mu(\gcd(q_0, n))
	\varphi(\gcd(q_0, n)), 
\end{multline*}
since $q_0$ is squarefree and $\gcd(q_0, q_*)=1$. 
Finally, since $q$ has the same prime factors as $q_*q_0$, we have
$\frac{\varphi(q)}{\varphi(q_* q_0)}=\frac{q}{q_* q_0}=q_2$.
\end{proof}

\begin{lemma}\label{lem:L_twist_Rsum}
Let $\xi\pmod*{N}$ and $\chi\pmod*{q}$ be Dirichlet characters,
with $(q,N)=1$.
Let $\{f_n\}_{n=1}^\infty$ be a sequence of complex numbers
of at most polynomial growth, and define $\Lambda_f(s)$ and
$\Lambda_f(s,c_\chi)$ as in \eqref{eq:Lambdaf} and
\eqref{e:Lambda}. Suppose that $f_1=1$ and the $f_n$ 
satisfy the Hecke relations at primes
not dividing $N$, so that
\begin{equation}\label{e:L_euler}
\Lambda_f(s)=\Gamma_\C(s+\tfrac{k-1}2)
	\sum_{n\mid N^\infty}\lambda_nn^{-s}
	\prod_{p\nmid N} \bigl(1-\lambda_pp^{-s}+\xi(p)p^{-2s}\bigr)^{-1},
\end{equation}
where $\lambda_n:=f_nn^{-\frac{k-1}2}$.
Let $\chi_*\pmod*{q_*}$ be the primitive character inducing $\chi$,
and define
$D_{f,\chi}(s)=\Lambda_f(s,c_\chi)/\Lambda_f(s, c_{\chi_*})$.
Then $D_{f,\chi}(s)$ is a Dirichlet polynomial given by the following
formula:
\begin{multline}\label{e:L_twist_Rsum}
	D_{f,\chi}(s)
	=
	\prod_{p\mid q_*}\lambda_{p^{\ord_p(q/q_*)}}p^{\ord_p(q/q_*)(1-s)}
	\\
	\times\prod_{p\mid q, p\nmid q_*}
	p^{(\ord_p(q)-1)(1-s)}
	\bigg[
	\lambda_{p^{\ord_p(q)}} p^{1-s}
	+\lambda_{p^{\ord_p(q)-2}}\xi(p) p^{-s} 
	-\lambda_{p^{\ord_p(q)-1}}\bigl(\chi_*(p)
	+\xi(p)\overline{\chi_*(p)}p^{1-2s}\bigr)
	\bigg]
	,
\end{multline}
where we define $\lambda_{p^\ell}=0$ for any negative integer $\ell$. 

Suppose further that $\{g_n\}_{n=1}^\infty$
is a sequence of at most polynomial growth such that $g_1\ne0$,
$g_n=g_1\overline{\xi(n)}f_n$ for all $n$ coprime to $N$, and
$$
\Lambda_g(s)=g_1\Gamma_\C(s+\tfrac{k-1}2)
	\sum_{n\mid N^\infty}\tilde{\lambda}_nn^{-s}
	\prod_{p\nmid N} \bigl(1-\tilde{\lambda}_pp^{-s}
	+\overline{\xi(p)}p^{-2s}\bigr)^{-1},
$$
where $\tilde{\lambda}_n=g_1^{-1}g_nn^{-\frac{k-1}2}$.
Then $D_{f,\chi}(s)$ and
$D_{g,\overline\chi}(s):=
\Lambda_g(s,c_{\overline\chi})/\Lambda_g(s,c_{\overline{\chi}_*})$
satisfy the functional equation
\begin{equation}\label{eq:feDfchi}
	D_{f, \chi}(s)
	=
	(q/q_*)^{1-2s}\xi(q/q_*)
	D_{g,\overline\chi}(1-s).
\end{equation}
In particular, if
$\Lambda_f(s,\overline{\chi}_*)$ and $\Lambda_g(s,\chi_*)$
satisfy \eqref{eq:femult} with $(\overline{\chi}_*,q_*)$
in place of $(\chi,q)$, then $\Lambda_f(s,c_\chi)$ and
$\Lambda_g(s,c_{\overline\chi})$ satisfy \eqref{eq:fe_cchi}.
\end{lemma}

\begin{proof}
Let $q_0 = \prod_{p\mid q, p\nmid q_*} p$ and $q_2 = \frac{q}{q_0 q_*}$. 
By \eqref{e:cchi_decomp}, we have
\begin{multline*}
	\frac{\Lambda_f(s, c_\chi)}{\Gamma_\C\!\left(s+\frac{k-1}{2}\right)}
	=
	\sum_{n=1}^\infty \frac{\lambda_{nq_2} c_\chi(nq_2)}{(nq_2)^s}
	\\
	=
	q_2\chi_*(q_0) \tau(\chi_*)
	\mu(q_0)
	\sum_{n=1}^\infty \frac{\lambda_{nq_2} \overline{\chi_*(n)} \mu(\gcd(q_0, n)) \varphi(\gcd(q_0, n))}
	{(nq_2)^s}
	\\
	=
	q_2\chi_*(q_0) \tau(\chi_*)
	\mu(q_0)
	\sum_{n\mid N_\infty}\frac{\lambda_n \overline{\chi_*(n)}}{n^s}
	\prod_{p\nmid q N} \sum_{j=0}^\infty \frac{\lambda_{p^j} \overline{\chi_*(p^j)}}{p^{js}}
	\prod_{p\mid\gcd(q_2, q_*)} \frac{\lambda_{p^{\ord_p(q_2)}}}{p^{\ord_p(q_2) s}}
	\\
	\times
	\prod_{p\mid q_0}
	\chi_*(p^{\ord_p(q)-1})
	\bigg[\frac{\lambda_{p^{\ord_p(q)-1}}\overline{\chi_*(p^{\ord_p(q)-1})}}{p^{(\ord_p(q)-1)s}}
	-
	\varphi(p) \sum_{j=\ord_p(q)}^\infty \frac{\lambda_{p^{j}} \overline{\chi_*(p^{j})}}
	{p^{js}}
	\bigg].
\end{multline*}
Thus,
\begin{multline*}
	D_{f, \chi}(s)
	=
	\frac{\Lambda_f(s, c_\chi)}{\Lambda_f(s, c_{\chi_*})}\\
	=
	q_2
	\prod_{p\mid q_*} \frac{\lambda_{p^{\ord_p(q/q_*)}}}{p^{\ord_p(q/q_*)s}}
	\prod_{p\mid q_0}
	\chi_*(p^{\ord_p(q)})
	\frac{-\frac{\lambda_{p^{\ord_p(q)-1}}\overline{\chi_*(p^{\ord_p(q)-1})}}{p^{(\ord_p(q)-1)s}}
	+
	\varphi(p) \sum_{j=\ord_p(q)}^\infty \frac{\lambda_{p^{j}} \overline{\chi_*(p^{j})}}
	{p^{js}}}
	{(1-\lambda_p \overline{\chi_*(p)} p^{-s} + \xi\cdot\overline{\chi_*}^2 (p) p^{-2s})^{-1}}.
\end{multline*}
For each prime $p\mid q_0$, we have
\begin{multline*}
	-\frac{\lambda_{p^{\ord_p(q)-1}}\overline{\chi_*(p^{\ord_p(q)-1})}}{p^{(\ord_p(q)-1)s}}
	+
	\varphi(p) \sum_{j=\ord_p(q)}^\infty \frac{\lambda_{p^{j}} \overline{\chi_*(p^{j})}}
	{p^{js}}
	\\
	=
	-\frac{\lambda_{p^{\ord_p(q)-1}}\overline{\chi_*(p^{\ord_p(q)-1})}}{p^{(\ord_p(q)-1)s}}
	-
	\varphi(p) \sum_{j=0}^{\ord_p(q)-1} \frac{\lambda_{p^{j}} \overline{\chi_*(p^{j})}}
	{p^{js}}
	+
	\varphi(p) (1-\lambda_p \overline{\chi_*(p)} p^{-s} + \xi\cdot\overline{\chi_*}^2 (p) p^{-2s})^{-1}.
\end{multline*}
Since $\lambda_{p^j} \lambda_p = \lambda_{p^{j+1}} + \xi(p) \lambda_{p^{j-1}}$, 
we have
\begin{multline*}
	\sum_{j=0}^{\ord_p(q)-2} \frac{\lambda_{p^{j}} \overline{\chi_*(p^{j})}}
	{p^{js}}
	=
	\bigg[
	\frac{\lambda_{p^{\ord_p(q)-2}} \overline{\chi_*(p^{\ord_p(q)-2})} \xi\cdot \overline{\chi_*}^2(p) p^{-s}}
	{p^{(\ord_p(q)-1)s}}
	-
	\frac{\lambda_{p^{\ord_p(q)-1}} \overline{\chi_*(p^{\ord_p(q)-1})}}
	{p^{(\ord_p(q)-1)s}}
	+
	1
	\bigg]
	\\
	\times
	(1-\lambda_p\overline{\chi_*(p)}p^{-s} +
	\xi\cdot\overline{\chi_*}^2(p)p^{-2s})^{-1},
\end{multline*}
so that
\begin{multline*}
	-\frac{\lambda_{p^{\ord_p(q)-1}}\overline{\chi_*(p^{\ord_p(q)-1})}}{p^{(\ord_p(q)-1)s}}
	+
	\varphi(p) \sum_{j=\ord_p(q)}^\infty \frac{\lambda_{p^{j}} \overline{\chi_*(p^{j})}}
	{p^{js}}
	=
	-p \frac{\lambda_{p^{\ord_p(q)-1}}\overline{\chi_*(p^{\ord_p(q)-1})}}{p^{(\ord_p(q)-1)s}}
	\\
	-
	\varphi(p)
	\bigg[
	\frac{\lambda_{p^{\ord_p(q)-2}} \overline{\chi_*(p^{\ord_p(q)-2})} \xi\cdot \overline{\chi_*}^2(p) p^{-s}}
	{p^{(\ord_p(q)-1)s}}
	-
	\frac{\lambda_{p^{\ord_p(q)-1}} \overline{\chi_*(p^{\ord_p(q)-1})}}
	{p^{(\ord_p(q)-1)s}}
	\bigg]
	\\
	\times
	(1-\lambda_p\overline{\chi_*(p)}p^{-s} + \xi\cdot\overline{\chi_*}^2(p)p^{-2s})^{-1}.
\end{multline*}
Therefore, for each prime $p\mid q_0$, we have
\begin{multline*}
	\frac{-\frac{\lambda_{p^{\ord_p(q)-1}}\overline{\chi_*(p^{\ord_p(q)-1})}}{p^{(\ord_p(q)-1)s}}
	+
	\varphi(p) \sum_{j=\ord_p(q)}^\infty \frac{\lambda_{p^{j}} \overline{\chi_*(p^{j})}}
	{p^{js}}}
	{(1-\lambda_p \overline{\chi_*(p)} p^{-s} + \xi\cdot\overline{\chi_*}^2 (p) p^{-2s})^{-1}}
	\\
	=
	\frac{\overline{\chi_*(p^{\ord_p(q)})}}{p^{(\ord_p(q)-1)s}}
	\bigg[
	\lambda_{p^{\ord_p(q)}}p^{1-s}
	-\lambda_{p^{\ord_p(q)-1}} \chi_*(p)
	+ 
	\lambda_{p^{\ord_p(q)-2}}\xi(p) p^{-s} 
	-\lambda_{p^{\ord_p(q)-1}} \xi\cdot\overline{\chi_*}(p)p^{1-2s}
	\bigg].
\end{multline*}
Writing $q_2=\prod_{p\mid q_*}p^{\ord_p(q/q_*)}
\prod_{p\mid q_0}p^{\ord_p(q)-1}$, this yields
\begin{multline*}
	D_{f, \chi}(s)
	=
	\prod_{p\mid q_*}\lambda_{p^{\ord_p(q/q_*)}}p^{\ord_p(q/q_*)(1-s)}
	\\
	\times\prod_{p\mid q, p\nmid q_*}
	p^{(\ord_p(q)-1)(1-s)}
	\bigg[
	\lambda_{p^{\ord_p(q)}}p^{1-s}
	-\lambda_{p^{\ord_p(q)-1}} \chi_*(p)
	+ 
	\lambda_{p^{\ord_p(q)-2}}\xi(p) p^{-s} 
	-\lambda_{p^{\ord_p(q)-1}} \xi\cdot\overline{\chi_*}(p)p^{1-2s}
	\bigg].
\end{multline*}

Since $\tilde{\lambda}_p = \overline{\xi(p)} \lambda_p$ for $p\mid q_0$, 
we also have 
\begin{multline*}
	(q/q_*)^{1-2s}\overline{\xi(q/q_*)}
	D_{f, \chi}(1-s)
	\\
	=
	q_2
	\prod_{p\mid q_*} 
	\overline{\xi(p^{\ord_p(q/q_*)})} p^{\ord_p(q/q_*)(1-2s)}
	\frac{\lambda_{p^{\ord_p(q/q_*)}}}{p^{\ord_p(q/q_*)(1-s)}}
	\\
	\times
	\prod_{p\mid q, p\nmid q_*}
	p^{-(\ord_p(q)-1)s}
	\overline{\xi(p^{\ord_p(q)})} \chi_*(p)
	\bigg[
	\lambda_{p^{\ord_p(q)}}\overline{\chi_*(p)} p^{1-s}
	-\lambda_{p^{\ord_p(q)-1}}p^{1-2s} 
	\\
	+ 
	\lambda_{p^{\ord_p(q)-2}}\overline{\chi_*(p)}\xi(p) p^{-s} 
	-\lambda_{p^{\ord_p(q)-1}} \xi\cdot\overline{\chi_*}^2(p)
	\bigg]
	\\
	=
	q_2
	\prod_{p\mid q_*} 
	\frac{\tilde{\lambda}_{p^{\ord_p(q/q_*)}}}{p^{\ord_p(q/q_*)s}}
	\prod_{p\mid q, p\nmid q_*}
	p^{-(\ord_p(q)-1)s}\overline{\chi_*(p)}
	\bigg[
	\tilde{\lambda}_{p^{\ord_p(q)}}\chi_*(p) p^{1-s}
	\\
	-\tilde{\lambda}_{p^{\ord_p(q)-1}} \overline{\xi(p)} \chi_*(p)^2 p^{1-2s} 
	+ 
	\tilde{\lambda}_{p^{\ord_p(q)-2}}\chi_*(p)\overline{\xi(p)} p^{-s} 
	-
	\tilde{\lambda}_{p^{\ord_p(q)-1}} 
	\bigg]
	\\
	=
	D_{g, \bar{\chi}}(s).
\end{multline*}
Finally, \eqref{eq:fe_cchi} follows from \eqref{eq:feDfchi} and
\eqref{eq:femult} (with $\chi$ replaced by $\overline{\chi}_*$) on noting
the equalities $c_{\chi_*}=\tau(\chi_*)\overline{\chi}_*$,
$c_{\overline{\chi}_*}=\tau(\overline{\chi}_*)\chi_*$ and
$\tau(\overline{\chi}_*)/\tau(\chi_*)
=q_*^{-1}\tau(\overline{\chi}_*)^2\chi_*(-1)$.
\end{proof}

\begin{lemma}\label{lem:Gamma1gen}
Let $\{g_1,\ldots,g_h\}$ be a generating set for
$\Gamma_1(N)$. For $i=1,\ldots,h$, let
$\gamma_i\in\langle T,W\rangle g_i\langle T,W\rangle$
be a matrix with top row $\sm r_i&b_i\esm$,
and choose $m_i\in\Z$ with $m_i\mid\frac{r_i-1}{N}$.
Then, for any $q\in\N$ satisfying $(q,Nm_i)=1$ and
$q\equiv Nm_ib_i\pmod*{r_i}$ for every $i$, we have
$H_q\supseteq\Gamma_1(N)$.
\end{lemma}
\begin{proof}
Fix a choice of $q$ satisfying the given conditions, and
set $d_i=(1-r_i)/(Nm_i)$. Then
$$
qd_i\equiv Nm_ib_id_i=(1-r_i)b_i\equiv b_i\pmod*{r_i}.
$$
By hypothesis we have $(q,Nm_i)=1$, so we can choose a matrix
$h_i\in\Gamma_0(N)$ with left column $\sm q\\Nm_i\esm$.
The upper-left entry of $\gamma_iT^{\frac{qd_i-b_i}{r_i}}h_i$
is $q(r_i+Nm_id_i)=q$, and thus $\gamma_iT^{\frac{qd_i-b_i}{r_i}}\in H_q$.
As shown in the proof of Theorem~\ref{thm:almostall},
$H_q$ also contains $T$ and $W$, and thus $g_i\in H_q$.
\end{proof}

\begin{lemma}\label{lem:height}
For $\gamma=\sm a&b\\Nc&d\esm\in\Gamma_0(N)$, define
$\height(\gamma)=\max\{|a|,|b|,|c|,|d|\}$.
Let $\tau_1,\ldots,\tau_\ell\in\bigl\{T,T^{-1},W,W^{-1}\bigr\}$, with
$\tau_{i+1}\ne\tau_i^{-1}$ for every $i=1,\ldots,\ell-1$.
Then, provided that $N\ge4$,
$$
\height(\tau_1\cdots\tau_\ell)\ge
\max\{\height(\tau_1\cdots\tau_{\ell-1}),
\height(\tau_2\cdots\tau_\ell)\}.
$$
\end{lemma}
\begin{proof}
Since $\height(\gamma)=\height(\gamma^{-1})$ for every $\gamma$, it
suffices to prove that
$\height(\tau_1\cdots\tau_\ell)\ge\height(\tau_1\cdots\tau_{\ell-1})$.
Suppose that this is false, and let $\tau_1,\ldots,\tau_\ell$ be a
counterexample of minimal length. Since
$\height(T^{\pm1})=\height(W^{\pm1})=\height(I)$, we must have $\ell>1$.

Note that $\langle T,W\rangle$ has some outer automorphisms that
preserve the height function. Specifically, conjugating an element
$\gamma=\tau_1\cdots\tau_\ell$ by $\sm 1&\\&-1\esm$
leaves $\height(\gamma)$ unchanged and swaps every occurrence of $T$
with $T^{-1}$ and $W$ with $W^{-1}$. Similarly, conjugating by
$\sm&-1\\N&\esm$ swaps $T$ with $W^{-1}$ and $W$ with $T^{-1}$.
Thus, applying an appropriate outer automorphism, we may assume without
loss of generality that $\tau_\ell=T$.

Write $\tau_1\cdots\tau_{\ell-1}=\sm a&b\\Nc&d\esm$.
Then by assumption we have
$h:=\height(\sm a&b\\Nc&d\esm)>\height(\sm a&b\\Nc&d\esm T)$, so that
$h=\max\{|a|,|b|,|c|,|d|\}>\max\{|a|,|a+b|,|c|,|Nc+d|\}$.
Hence, $h=\max\{|b|,|d|\}$. If $h=|b|$ then $|a|<|b|$ and $|a+b|<|b|$, so
$ab<0$. If $h=|d|$ then $|Nc+d|<|d|$, so $cd<0$ and $|Nc|<2|d|$.

Next we consider $\tau_{\ell-1}$, which must be one of
$T,W,W^{-1}$, since $\tau_\ell\ne\tau_{\ell-1}^{-1}$.
By minimality, we have
$\height(\sm a&b\\Nc&d\esm\tau_{\ell-1}^{-1})=
\height(\tau_1\cdots\tau_{\ell-2})\le h$.
If $\tau_{\ell-1}=T$ then we have
$\max\{|b-a|,|d-Nc|)\le h$, contradicting the fact
that $ab<0$ when $h=|b|$ and $cd<0$ when $h=|d|$.
If $\tau_{\ell-1}=W$ then we have
$\max\{|a-Nb|,|c-d|\}\le h$, which is again a contradiction.

Hence we may assume that $\tau_{\ell-1}=W^{-1}$, and we have
$\max\{|a+Nb|,|b|,|c+d|,|d|\}\le h$.
If $h=|b|$ then $|b|\ge|a+Nb|>(N-1)|b|$, which is a contradiction,
since $N>1$. Hence we must have $h=|d|$.

Next, let $j\in\{1,\ldots,\ell-1\}$ be the largest number
such that $\tau_{\ell-i}=W^{-1}$ for $i=1,\ldots,j$.
Since $|Nc|<2|d|$ and $N>1$, we must have $j<\ell-1$.
Consider $\tau_{\ell-j-1}$, which must be one of $T,T^{-1}$. We have
$$
\height(\sm a&b\\Nc&d\esm W^j\tau_{\ell-j-1}^{-1})=
\height(\tau_1\cdots\tau_{\ell-j-2})\le h.
$$
Since $\tau_{\ell-j-1}=T^{\pm1}$ and $jN\ge4$, this implies that
$$
|d|\ge\height(\sm a&b\\Nc&d\esm W^jT^{\mp1})
\ge|(jN\mp1)d+Nc|>(jN\mp1-2)|d|\ge|d|,
$$
which is a contradiction.
\end{proof}

For $N\ge4$, $\Gamma_1(N)$ is torsionfree \cite[Lemma~12.3]{kulkarni},
and hence free, by the Kurosh subgroup theorem \cite{kurosh}.
Lemma~\ref{lem:height} permits a simple, direct proof of the
following consequence:
\begin{corollary}\label{cor:TWfree}
$T$ and $W$ generate a free group if and only if $N\ge4$.
\end{corollary}
\begin{proof}
For $N\le3$, we verify directly that $(W^{-1}T)^{12}=I$.
For $N\ge4$, suppose that $\tau_1\cdots\tau_\ell=I$ is a nontrivial
relation of minimal length satisfied by $T$ and $W$. 
Clearly $\ell>1$, and by applying an appropriate outer automorphism,
we may assume that $\tau_1=T$. Considering each possible
$\tau_2\in\{T,W,W^{-1}\}$, we see that
$\height(\tau_1\tau_2)>1=\height(I)$,
in contradiction to Lemma~\ref{lem:height}.
\end{proof}
\bibliographystyle{amsplain}
\providecommand{\bysame}{\leavevmode\hbox to3em{\hrulefill}\thinspace}
\providecommand{\MR}{\relax\ifhmode\unskip\space\fi MR }
\providecommand{\MRhref}[2]{%
  \href{http://www.ams.org/mathscinet-getitem?mr=#1}{#2}
}
\providecommand{\href}[2]{#2}

\end{document}